\documentclass[reqno,12pt]{amsart}

\usepackage{amsmath, amssymb, amsthm, epsfig}
\usepackage{hyperref, latexsym}
\usepackage{url}
\usepackage[mathscr]{euscript}
\usepackage{color}
\usepackage{harpoon}
\usepackage{url}

\usepackage{fullpage} 
\usepackage{setspace}
\usepackage{refcheck} 
\norefnames
\nocitenames
\usepackage{mathtools}
\mathtoolsset{showonlyrefs}

\def\today{\ifcase\month\or
  January\or February\or March\or April\or May\or June\or
  July\or August\or September\or October\or November\or December\fi
  \space\number\day, \number\year}

\newtheorem{theorem}{Theorem}
\newtheorem*{problem}{Problem}
\newtheorem*{conjecture}{Conjecture}

\newcommand{\z}{\mathbb{Z}}

\renewcommand{\r}{\mathbb{R}}

\newcommand{\ft}{\widehat}
\renewcommand{\t}{\mathbb{T}}

\newcommand{\bo}{\boldsymbol}
\newcommand{\bx}{\bo x}

\newcommand{\bw}{\bo w}

\newcommand{\bu}{\bo u}
\newcommand{\bv}{\bo v}

\newcommand{\la}{\lambda}

\newcommand{\epp}{\epsilon}

\renewcommand{\leq}{\leqslant}
\renewcommand{\geq}{\geqslant}

\begin{document}


\title[]{Bounds for the lonely runner problem via linear programming}
\author[Gon\c{c}alves and Ramos]{Felipe Gon\c{c}alves and Jo\~{a}o P. G. Ramos}
\date{\today}
\subjclass[2010]{}
\address{Hausdorff Center for Mathematics, Universit\"at Bonn,
Endenicher Allee 60,
53115 Bonn, Germany }
\email{goncalve@math.uni-bonn.de}
\address{Department of Mathematics, ETH Z\"urich,
R\"amistrasse 101,
8092 Z\"urich, Switzerland} 
\email{joao.ramos@math.ethz.ch}

\allowdisplaybreaks


\begin{abstract}
In this note we develop a linear programming framework to produce upper and lower bounds for the lonely runner problem.
\end{abstract}


\maketitle


\section{The Lonely Runner Problem}

Suppose you are competing in race on a circular track of perimeter $L$ with $n-1$ other runners. Assume all competitors have distinct constant speeds. The {\it gap of loneliness} is the {\it largest} length $\ell$ such that at some time $t$ in the future (assuming the race continues forever) the closest runner to you is at distance $\ell$. The lonely runner conjecture states that 
$$
\ell \geq \frac{L}{n}.
$$ 
This problem was introduced independently by Wills \cite{W67} (1967) and Cusick \cite{C73} (1973) in the context of view obstruction problems. The conjecture is known to be true for $n\leq 7$ runners. Moreover, speeds can be assumed to be distinct integers. For more on the history of partial results see Bohnman et al. \cite{BHK01} (2001), Perarnau \& Serra \cite{PS16} (2016) and Tao  \cite{T18} (2018). 

By Galilean relativity your speed can be assumed to be zero and the conjecture takes the following equivalent formulation: Let $\|x\|=\min_{n\in\z} |x-n|$ denote the distance to the nearest integer. For a vector $\bx \in \r^{n-1}$ let
$$
\mu(\bx)=\min\{\|x_1\|,...,\|x_{n-1}\|\}.
$$
Then for any vector $\bv=(v_1,...,v_{n-1})\in \z^{n-1}$ of distinct integers show that
$$
gap(\bv):=\max_{t\in \t} \mu(t\bv)\geq \frac{1}{n}, \ \ \ (\text{\it conjecture})
$$
where $\t=\r/\z$. 

In general, since $t\mapsto \mu(t\bv)$ is piece-wise linear with slopes drawn from $\{v_1,...,v_{n-1}\}$, the set of local maxima of $t\mapsto \mu(t\bv)$ is contained in the intersection of any two line segments. Therefore, if $t$ satisfies $\mu(t\bv)=gap(\bv)$ then $t\in\tfrac{1}{q}\z$, where $q$ is a factor of some $(v_j\pm v_i)$ with $i\neq j$, and therefore $gap(\bv)=a/b$ where $b$ divides $q$.



\subsection{The Linear Programming Approach}
We want to study the following two problems:

\begin{problem}[Linear Programming Problems]\label{prob:lp}
Fix a sign $\epsilon={\pm 1}$. Let $\bv=(v_1,...,v_{n-1})$ be a given vector of increasing integer speeds. We want to
\begin{align*}
& \text{Minimize}\ \ \epsilon \, \frac{\ft f(0)}{f(0)}
\end{align*}
subjected to
\begin{enumerate}
\item[(I)] $f$ is a non-zero, even and real trigonometric polynomial
$$
f(x)=\sum_{n=-D}^D \ft f(k)e^{2\pi i k x} = \ft f(0) + 2\sum_{k=1}^D \ft f(k) \cos(2\pi i kx)
$$
of degree at most $D\geq \max (v_1,...,v_{n-1})$.
\item[(II)] In case $\epsilon=+1$ we ask
$$
f(x)\geq 0 \ \ \text{for} \ \ 0\leq x\leq \tfrac12
$$
and in case $\epsilon=-1$ we ask 
$$
f(x)\leq 0 \ \ \text{for} \ \ \tfrac1{v_{n-1}+v_{n-2}}\leq x\leq \tfrac12;
$$
\item[(III)] $\epsilon\ft f(k)\leq 0$ if $k \notin\{0,v_1,...,v_{n-1}\}$.
\end{enumerate}
We denote by $\Lambda_\epp(\bv)$ the class of trigonometric polynomials satisfying (I),(II) and (III). We write
$$
\la_+(\bv)=\inf_{f\in \Lambda_+(\bv)} \frac{\ft f(0)}{f(0)} \ \ \ \text{and} \ \ \ \la_-(\bv)=\sup_{f\in \Lambda_-(\bv)} \frac{\ft f(0)}{f(0)}
$$
\end{problem}

\begin{theorem}\label{thm:main0}
Let $\bv=(v_1,...,v_{n-1})$ be a vector of increasing positive integers. Then
\begin{equation}\label{eq:mainineq+}
gap(\bv) \leq \la_+(\bv) 
\end{equation}
and
\begin{align}\label{eqmainineq-}
gap(\bv) \geq \la_-(\bv).
\end{align}
\end{theorem}

\begin{theorem}\label{thm:main+}
Let $\bv=(v_1,...,v_{n-1})$ be a vector of increasing positive integers. Then equality is attained in \eqref{eq:mainineq+} if one of the following conditions hold:
\begin{enumerate}
\item[(i)] All $v_i$'s are odd. In this case $f(x)=\cos(\pi v_i x)^2$ is optimal for any $i=1,...,n-1$;
\item[(ii)] There exist coprime integers $a,m\geq 1$ such that $a\{1,...,m-1\}\subset \{v_1,...,v_{n-1}\}$ and all integers in $\{v_1,...,v_{n-1}\}\setminus a\{1,...,m-1\}$ are not divisible by $m$. In this case $f(x)=K_m(ax)$ is optimal, where $K_m$ is Fejer's kernel \eqref{def:fejerker}.
\item[(iii)] There exists integer $a\geq 1$ such that $\bv=a\bv'$ and $\bv'$ satisfies condition (i) or (ii). In this case if $f(x)$ is optimal for $\bv'$ then $f(ax)$ is optimal for $\bv$. 
\end{enumerate}
\end{theorem}

In the range $n\leq 20$ and $\max(v_1,...,v_n)\leq 40$ we have performed a computer search in $\bv$ in conjunction with Gurobi's linear programming solver \cite{gurobi} to approximate $\la_+(\bv)$.  The sign conditions of $f$ was modelled with sampling. This produced reliable numerical approximations to what we believe is the true value of $\la_+(\bv)$. In this way we check that the only cases where $$gap(\bv)+[\text{very small error}] > \la_+(\bv)$$ for $|\bv|_\infty \leq 40$ and $n\leq 20$ were the ones contemplated by Theorems \ref{thm:main+}. This leads to the following conjecture.

\begin{conjecture}
Equality is attained in \eqref{eq:mainineq+} if and only if one of the conditions in Theorem \ref{thm:main+} hold.
\end{conjecture}

It is unfortunate that the bounds generated for $\la_-(\bv)$ do not seem to be nearly as good as the bounds generated by $\la_+(\bv)$, and we believe this is because condition (II) seems to be very strong for $\epp=-1$. We need high degree polynomials and a large number of sampling points for feasibility of the linear program. In Section \ref{sec:lb} we propose an improved version of this lower bound.

We note that proving exact bounds is not hard as if some numerical $f$ satisfies $\epp f \geq -\delta$ in some region, but $\epp f$ should be nonnegative in that region, then all we have to do is use $g = f+\epp\delta$ as this would be admissible for $\Lambda_\pm(\bv)$ and $\ft {g}(0)/g(0) = {\ft f(0)}/{f(0)}+O(\delta)$.

\section{Proofs for the main results}
We start by recalling that Dirichlet's Approximation Theorem implies the Lonely Runner Conjecture is sharp; a rephrasing of Dirichlet's theorem is
$$
\max_{t\in\t}{\mu(t(1,2,...,n-1))} = \frac{1}{n}.
$$
The maxima is attained for $t=a/n$ for $a$ coprime with $n$. We now observe that inequality \eqref{eq:mainineq+} is also tight in case $\bv=(1,2,...,n-1)$, and Fej\'er's kernel 
\begin{align}\label{def:fejerker}
K_n(x)=\frac{1}{n}\left(\frac{\sin(\pi n x)}{\sin(\pi x)}\right)^2 =\sum_{|j|<n-1} (1-|j|/n)e^{2\pi i j x}
\end{align}
is the unique optimum. Optimality can easily be checked by hand, while uniqueness (modulo scaling) comes from the proof of Theorem \ref{thm:main0}. Essentially, because $f\geq 0$, we must have $f(k/n)=f'(k/n)=0$ for $k=1,...,n$ and Fej\'er's kernel is the only even trigonometric polynomial of degree $n-1$ with these properties.


\begin{proof}[\bf Proof of Theorem \ref{thm:main0}]
Let $\delta=gap(\bv)=\mu(t\bv)$ and $h(x)=(\delta-|x|)_+$ be a hat function. Since $\ft h(x) = (\sin(\pi \delta x)/(\pi x))^2$ we have
$$
h(x)=\sum_{n\in\z} \left(\frac{\sin(\pi \delta n)}{\pi n}\right)^2 e^{2\pi i n x}.
$$
If $f\in \Lambda_+(\bv)$ we obtain
\begin{align*}
\begin{split}
\delta\ft f(0) =\delta \ft f(0) + 2\sum_{j=1}^{n-1} \ft f(v_j) h(t v_j)   \geq \sum_{k=-D}^D \ft f(k) h(t k) = \sum_{k\in\z} \left(\frac{\sin(\pi \delta k)}{\pi k}\right)^2 f(kt) \geq \delta^2 f(0)
\end{split}
\end{align*}
which proves the upper bound. For the lower bound ($\epp=-1$), first recall that $t=p/q$, $\delta=a/b$ and $b$ divides $q$, while $q$ is a factor of some $v_j \pm v_i$ with $j>i$ (both fractions in lowest terms). In particular $q\leq v_{n-1}+v_{n-2}$. If $f\in \Lambda_-(\bv)$ we obtain
\begin{align*}
\begin{split}
\delta\ft f(0) =\delta \ft f(0) + 2\sum_{j=1}^{n-1} \ft f(v_j) h(t v_j)  & \leq \sum_{k=-D}^D \ft f(k) h(t k) \\ & = \delta^2 f(0)+ \sum_{k\in\z\setminus {b\z}} \left(\frac{\sin(\pi \tfrac{a}{b} k)}{\pi k}\right)^2 f(k\tfrac{p}{q})
\end{split}
\end{align*}
However, since $\{kp/q \mod 1: k\in\z_+\setminus {b\z}\}=\tfrac1{q}\{1,...,q-1\}\setminus b\z$, $q\leq v_{n-1}+v_{n-2}$ and $f(x)\leq 0$ for $1/(v_{n-1}+v_{n-2}) \leq x\leq 1/2$, then $f(kp/q)\leq 0$ for $k\in\z\setminus {b\z}$. This concludes the proof.
\end{proof}

This proof is inspired by the analytic proof of Dirichlet's approximation theorem due to Montgomery \cite{M94} (1994). We observe that equality is attained in \eqref{eq:mainineq+} or \eqref{eqmainineq-} if and only if there is $f\in \Lambda_s(\bv)$ such that:
\begin{enumerate}
\item[(a)] $\ft f(k)= 0$ if $k \notin\{v_1,...,v_{n-1}\}$;
\item[(b)] For some $t=p/q$ that is a global maxima of $t\mapsto \mu(t\bv)$, where $gap(\bv)=a/b$ and $b$ divides $q$ (both fractions in lowest terms) we have $f(k p/q)=0$ if $k\in \{1,...,q-1\}\setminus b \z$.
\end{enumerate}

\begin{proof}[\bf Proof of Theorem \ref{thm:main+}]
Condition (i) is easy to check because when all $v_i$'s are odd we have $gap(\bv)=1/2$ while $\ft f(0)=1/2$ and $f(0)=1$ for $f(x)=\cos(\pi v_i x )^2$. Next, if condition $(ii)$ holds, then letting $\bu=(1,...,m-1)$ and $\bw=\{\bv\}\setminus \{a\bu\}$ (abusing notation) we obtain
$$
\mu(t \bv) = \min(\mu(t a \bu),\mu(t \bw)) \leq \mu(t a \bu) \leq \frac{1}{m}.
$$
On the other hand, since $a$ is coprime with $m$ we have $\mu(\tfrac1m a \bu)=\tfrac{1}{m}$, and since $\tfrac{1}m\bw$ has no integer coordinate we have $\mu(\tfrac{1}m\bw)\geq \tfrac1m$. Therefore $\mu(\tfrac{1}m\bv)=\tfrac1m$ and we obtain that $gap(\bv)=\tfrac1m$. Now it is easy to check that $f(x)=K_m(ax)$ belongs to $\Lambda_+(\bv)$ and is optimal. Condition (iii) is trivial.
\end{proof}

\section{Improved lower bounds}\label{sec:lb}
Let $V_{q}$ be the class of vectors $\bv$ of increasing positive integers such that the global maxima of $t\mapsto \mu(t\bv)$ is attained at some point $t=p/q\in (0,\tfrac12)$ (in lowest terms). By the proof of Theorem \ref{thm:main0}, we see that if we let $\Lambda_-(\bv,q)$ be the class of functions satisfying the above conditions (I), (III) and
$$
\text{(II')}\ f(x)\leq 0 \ \ \text{for} \ \ \tfrac{1}{q}\leq x\leq \tfrac12,
$$
then 
\begin{equation}\label{eq:mainineq-}
\la_-(\bv,q):=\sup_{f\in \Lambda_-(\bv,q)} \frac{\ft f(0)}{f(0)} \ \leq \ gap(\bv)
\end{equation}
for any $\bv\in V_{q}$.  For instance $(1,...,n-1)\in V_{n}$, but some other examples of vectors in $V_n$ can be extracted from Goddyn and Wong \cite{GW06}, (2006). They present conditions for $\bv$ to be {\it tight}, that is, $gap(\bv)=\tfrac{1}{n}$. Some of these tight vectors characterized in \cite[Theorem 2.3]{GW06} belong to $V_n$, for instance: $$(1,2,...,n-3,n-1,2n-4)\in V_{n} \ \ \text{if} \ \ n=2 \ \ (\text{mod}\ 2\cdot 3)$$
$$(1,2,...,n-4,n-2,n-1,2n-6)\in V_{n} \ \ \text{if} \ \ n=3 \ \ (\text{mod}\ 2\cdot 3 \cdot 5).$$

\begin{theorem}\label{thm:main-}
Let $\bv=(v_1,...,v_{n-1})\in V_{q}$. Then equality is attained in \eqref{eq:mainineq-} if one of the following conditions hold:
\begin{enumerate}
\item[(i)] There exist integer $a\geq 1$ coprime with $q$ such that $a\{1,...,q-1\}\subset \{v_1,...,v_{n-1}\}$ and all integers in $\{v_1,...,v_{n-1}\}\setminus a\{1,...,q-1\}$ are not divisible by $q$. In this case 
$$
f(x)={K_q(ax)}\frac{1-\cos(\pi/q)^2}{(\cos(\pi a x)^2-\cos(\pi/q)^2)}
$$
is optimal, where $K_q$ is Fejer's kernel \eqref{def:fejerker}.
\item[(ii)] There exists an integer $a\geq 1$ such that $\bv=a\bv'$ and $\bv'$ satisfies condition (i). In this case if $f(x)$ is optimal for $\bv'$ then $f(ax)$ is optimal for $\bv$. 
\end{enumerate}
\end{theorem}

\begin{proof}
Assume condition (i). By the same discussion in the proof of Theorem \ref{thm:main+} we have $gap(\bv)=1/q$. It is easy to see that $f$ satisfies conditions (I),(II'),(III). Moreover its mass equals the mass of $f(x/a)$, which in turn, by exact Gaussian quadrature, equals the mass of $K_q(x)$, which is $1$. Also $f(0)=q$. This shows optimality. Condition (ii) is trivial.
\end{proof}

\subsection{Numerics}
We used Gurobi \cite{gurobi} to compute $\la_\pm(\bv)$ and $\la_-(q,\bv)$ for $2429$ different velocity vectors $\bv=(v_1,...,v_5)$ selected randomly from $0<v_1<...<v_5\leq 50$. In Figure \ref{fig} we plot points $(x,y)$ which are numerical approximations to $(\la_-(\bv,q),gap(\bv))$ in blue dots, $(\la_-(\bv),gap(\bv))$ in yellow triangles and $(\la_+(\bv),gap(\bv))$ in green squares. We took $q=\text{denominator}(t_{max})$ where $\mu(t_{max}\bv)=gap(\bv)$, and $t_{max}$ is the smallest with this property. The diagonal blue line is $x=y$ and the gray vertical line is $x=1/6$. We note that we get much better lower bounds if we know a priori that $\bv\in V_q$, as the yellow triangles are clearly much closer to the line $x=y$ than the blue dots. The plot appears to have some interesting emergent structures: rays of triangles and parabola-like green structures.

\begin{figure}[!h]
\begin{center}
\includegraphics[scale=1.1]{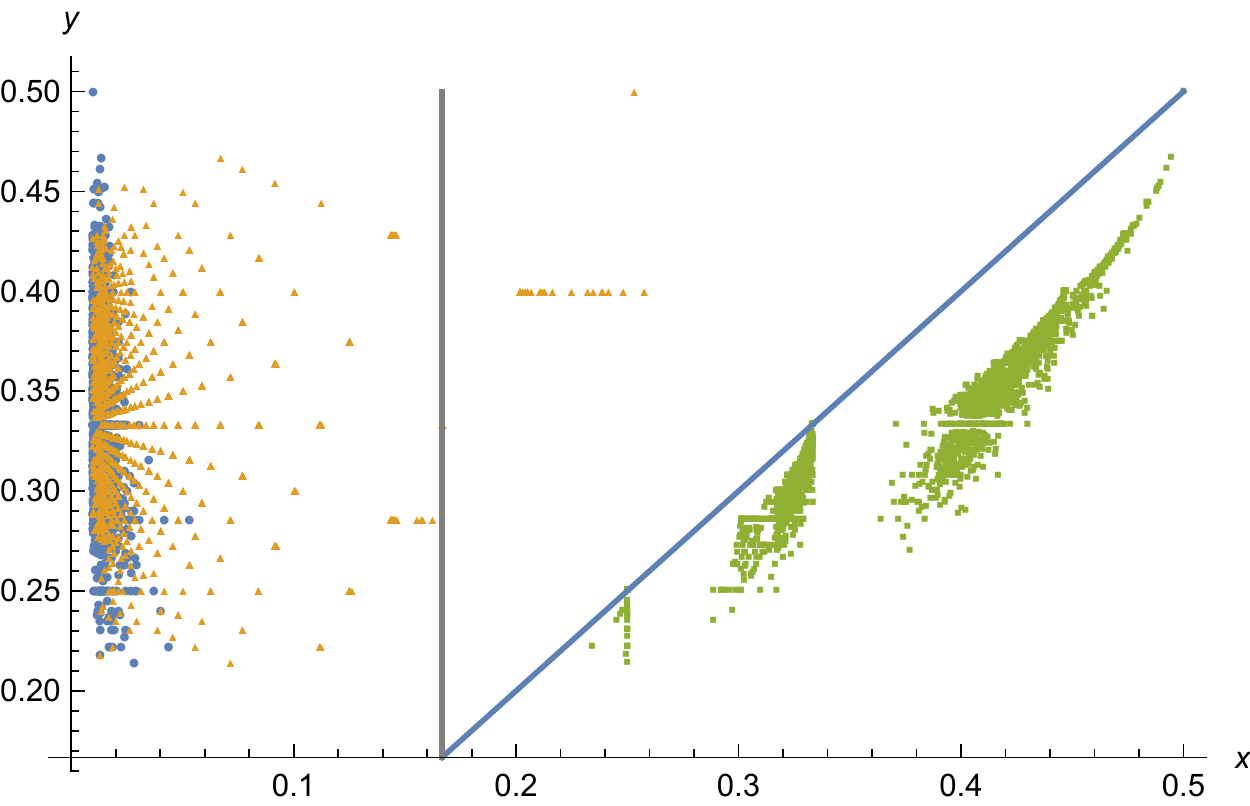}
\caption{}
\label{fig}
\end{center}
\end{figure}

\section*{Acknowledgements}
The first author is thankful to Jeffrey Vaaler for helpful comments. 
F.G.\@ acknowledges support from the Deutsche Forschungsgemeinschaft through the Collaborative Research Center 1060.


\begin{thebibliography}{100}
 \bibitem{BHK01}
 T. Bohman, R. Holzman, D. Kleitman, Six lonely runners, In honor of Aviezri Fraenkel on the occasion of his 70th birthday. Electron. J. Combin. 8 (2001), no. 2, Research Paper 3, 49 pp.
 
 \bibitem{C73}
T. W. Cusick, View obstruction problems, Aequationes Math. 9 (1973), 165–170.
 
 \bibitem{GW06} 
 L. Goddyn, E. B. Wong, Tight instances of the lonely runner, Integers 6 (2006), A38.
 
 \bibitem{gurobi}
{Gurobi Optimization, LLC},
\newblock {\it Gurobi Optimizer Reference Manual} (2020).

 
\bibitem{M94}
H. L. Montgomery,
Ten lectures on the interface between analytic number theory and harmonic analysis,
CBMS Regional Conference Series in Mathematics, 84 (1994).
 
 \bibitem{PS16}
G. Perarnau, O. Serra, Correlation among runners and some results on the lonely runner conjecture, Electron. J. Combin. 23 (2016), no. 1, Paper 1.50, 22 pp.


\bibitem{T18}
T. Tao,
Some remarks on the lonely runner conjecture, 
Contrib. Discrete Math. 13 (2018), no. 2, 1–31.

\bibitem{W67}
 J. M. Wills, Zwei S\"atze \"uber inhomogene diophantische Approximation von Irrationalzahlen, Monatsch. Math. 61 (1967), 263-269.
 

 


\end{thebibliography}
\end{document}